\documentclass[a4paper,11pt]{amsart}
\usepackage{hyperref}
\usepackage[all]{xy}
\usepackage{enumitem}
\usepackage[normalem]{ulem}
\newcommand{\tsk}[1]{\textcolor{YellowOrange}}
\usepackage{tikz}
\usepackage{tikz-cd}

\makeatletter
\def\@endtheorem{\endtrivlist}% NEW
\makeatother
\usepackage[active]{srcltx}
\usepackage{amsmath}
\usepackage{amssymb}
\usepackage{amscd}
\usepackage{amsthm}
\usepackage[latin1]{inputenc}
\usepackage{mathrsfs}  

\usepackage{nicefrac}

\newtheorem{teo}{Theorem}[section]
\newtheorem{defin}[teo]{Definition}
\newtheorem{prop}[teo]{Proposition}
\newtheorem{cor}[teo]{Corollary}
\newtheorem{lemma}[teo]{Lemma}
\newtheorem{expectation}[teo]{Expectation}
\theoremstyle{definition}

\newtheorem{remark}[teo]{Remark}

\newtheoremstyle{dico}% name of the style to be used
 {\baselineskip}   % ABOVESPACE
  {\topsep}   % BELOWSPACE
  {}  % BODYFONT
  {0pt}       % INDENT (empty value is the same as 0pt)
  {} % HEADFONT
  {.}         % HEADPUNCT
  {5pt plus 1pt minus 1pt} % HEADSPACE
  {}          % CUSTOM-HEAD-SPEC
\theoremstyle{dico}

\numberwithin{equation}{section}

\newcommand{\ra}{\rightarrow}
\newcommand{\C}{\mathbb{C}}
\newcommand{\R}{\mathbb{R}}
\newcommand{\Zeta}{{\mathbb{Z}}}
\newcommand{\N}{{\mathbb{N}}}
\newcommand{\QQ}{{\mathbb{Q}}}
\newcommand{\meno}{^{-1}}

\newcommand{\alfa}{\alpha}

\newcommand{\vacuo}{\emptyset}

\newcommand{\La}{\Lambda}

\newcommand{\enf}{\emph}
  
\newcommand{\desudt}[1] []      {\dfrac {\mathrm {d} #1 }{\mathrm {dt}}}
\newcommand{\desudtzero}        {\desudt \bigg \vert _{t=0} }

\newcommand{\restr}[1]          {\vert_{#1}}
\newcommand{\Aut}{\operatorname{Aut}}

\newcommand{\End}{\operatorname{End}}

\newcommand{\lieg}{\mathfrak{g}}

\newcommand{\lieh}{\mathfrak{h}}

\newcommand{\liem}{\mathfrak{m}}

\newcommand{\Ad}{\operatorname{Ad}}

\newcommand{\om}{\omega}
\newcommand{\eps}{\varepsilon}
\renewcommand{\phi}{\varphi}
\newcommand{\lds}{\ldots}

\newcommand{\cd}{\cdot}

\newcommand{\sx}{\langle}

\newcommand{\lra}{\longrightarrow}

\newcommand{\ga}{\gamma}
\newcommand{\Ga}{\Gamma}

\newcommand{\id}{\operatorname{id}}

\newcommand{\Gl}{\operatorname{GL}}
\newcommand{\GL}{\operatorname{GL}}

\newcommand{\PP}{\mathbb{P}}   
   
\renewcommand{\phi}             {\varphi}

\newcommand{\sieg}{\mathfrak{S}}

\newcommand{\siegv}{\sieg(V,\om)}

\newcommand{\M}{\mathsf{M}}
\newcommand{\tM}{\widetilde{\mathsf{M}}}
\newcommand{\HE}{\mathsf{HE}}
\newcommand{\Tg}{\mathsf{T}_g}
\newcommand{\T}{\mathsf{T}}
\newcommand{\Mg}{\mathsf{M}_g}
\newcommand{\A}{\mathsf{A}}

\newcommand{\Ag}{\mathsf{A}_g}

\newcommand{\teich}{\mathscr{T}}

\newcommand{\Sl}                {\operatorname {SL}}
\newcommand{\SL}                {\operatorname {SL}}

\newcommand{\Sp}                {\operatorname {Sp}}

\renewcommand{\Im}              {\operatorname{Im}}

\newcommand{\grass}             { \operatorname{\mathbb {G}} }

\newcommand{\ag}{\mathsf{A}_g}

\newcommand{\Datum}{\Delta}
\newcommand{\datum}{{( G, \theta)}}
\newcommand{\ada}{\mathsf{P}_\Datum}
\newcommand{\md}{\mathsf{M}_\Datum}
\newcommand{\tmd}{\tM_\Datum}

\newcommand{\Map}{\operatorname{Map}}

\newcommand{\co}{\mathscr{C}}
\newcommand{\Aff}{\operatorname{Aff}}
\newcommand{\seff}{B}
\newcommand{\mihi}[1]{}
\newcommand{\prym}{\mathscr{P}}
\newcommand{\Nm}{\operatorname{Nm}}

\begin{document}

\author{Paola Frediani, Alessandro Ghigi, Irene Spelta}
\title[Infinitely many Shimura varieties]
  {Infinitely many Shimura varieties in the Jacobian locus for  $g \leq 4$}

\address{Universit\`{a} di Pavia}
\email{paola.frediani@unipv.it}
\email{alessandro.ghigi@unipv.it}
\email{irene.spelta01@ateneopv.it} 

\thanks{The authors were partially supported by MIUR PRIN 2017
  ``Moduli spaces and Lie Theory'' ,  by MIUR, Programma Dipartimenti di Eccellenza
  (2018-2022) - Dipartimento di Matematica ``F. Casorati'',
  Universit\`a degli Studi di Pavia and by INdAM (GNSAGA).  }

\maketitle

\begin{abstract}
  We study families of Galois covers of curves of positive genus. It
  is known that under a numerical condition these families yield
  Shimura subvarieties generically contained in the Jacobian locus.
  We prove that there are only 6 families satisfying this condition,
  all of them in genus 2,3 or 4.  We also show that these families admit
  two fibrations in totally geodesic subvarieties, generalizing a
  result of Grushevsky and M\"oller.  Countably many of these fibres
  are Shimura. Thus the Jacobian locus contains infinitely many
  Shimura subvarieties of positive dimension of any $g  \leq 4$.
\end{abstract}

\tableofcontents{}

\section{Introduction}

Denote by $\M_g$ the moduli space of curves, by $\Ag$ the moduli space
of principally polarized abelian varieties and by $j : \M_g \ra \A_g$
the period map.  The \emph{Torelli locus} $\Tg$ is the closure of
$j(\Mg)$ in $\Ag$.  A \emph{special} or \emph{Shimura subvariety} of
$\A_g$ is by definition a Hodge locus for the tautological family of
principally polarized abelian varieties on $\A_g$.  A subvariety
$Z \subset \A_g$ is \emph{generically contained in} $j(\Mg)$ if
$Z \subset \T_g$ and and $ Z \cap j(\M_g) \neq \vacuo$.
\begin{expectation}[Coleman-Oort]
  For large $g$ there are no special subvarieties of positive
  dimension generically contained in $j(\Mg)$.
\end{expectation}
(See \cite{coleman,oort-can} and \cite{moonen-oort} for a thorough
survey.  See \cite{cfg,cfgp,dejong-zhang,fgp,fpp,fpi,
  fp,deba,gm1,gpt,hain,liu-yau-ecc,lz,moonen-special} for related
results.)

Shimura subvarieties
are \emph{totally geodesic}, i.e they are images of totally geodesic
submanifolds of the Siegel space, that we denote by $\sieg_g$.  More
precisely, by results of Mumford and Moonen an algebraic subvariety of
$\A_g$ which is totally geodesic is a Shimura subvariety if and only
if it contains a CM point, see
\cite{mumford-Shimura,moonen-linearity-1}.  While the notion of CM
point is arithmetic, the condition of being totally geodesic relates
to the locally symmetric geometry of $\A_g$ coming from the Siegel
space.  One expects the Torelli embedding to be very curved with
respect to this locally symmetric ambient geometry.  In particular
$\Tg$ should contain \emph{very few} totally geodesic subvarieties of
$\A_g$.

Yet there are at least \emph{some} Shimura (hence totally geodesic)
subvarieties contained in $\T_g$.  They are all in genus $g\leq 7$.
To describe them consider the following construction. Take a family of
Galois covers $C \ra C'=C/G$, where the genera $g(C')=g'$, $g(C) = g$,
the number of ramification points and the monodromy are fixed.  Let
$Z$ denote the closure in $\Ag$ of the locus described by $[JC]$ for
$C$ varying in the family.  The simple numerical condition
\begin{gather}
  \label{star}
  \tag{$\ast$} \dim (S^2 (H^0(K_C)))^G = \dim H^0(2K_C)^G
\end{gather}
is sufficient to ensure that $Z$ is Shimura \cite{fgp,fpp}.  Moonen
\cite{moonen-special} proved that when $g'=0$ and the group $G$ is
cyclic \eqref{star} is also necessary for $Z$ to be Shimura.  Mohajer
and Zuo \cite{mohajer-zuo-paa} extended this to the case where $g'=0$,
$G$ is abelian and the family is one--dimensional. In both cases the
authors also showed that condition \eqref{star}  holds only in the
known examples. These results are proved using methods from positive
characteristic.  A completely different Hodge theoretic argument was
given in \cite[Prop. 5.2]{cfg}, but only works for some of the
families of cyclic covers of $\PP^1$.  It is unknown whether
\eqref{star} is necessary in general for a family of covers to yield a
Shimura subvariety or whether other families exist which satisfy
\eqref{star}.

Shimura subvarieties via families of Galois covers of \emph{elliptic}
curves satisfying \eqref{star} were constructed in \cite{fpp}.  There
are 6 such families, but 4 yield Shimura subvarieties already gotten
using coverings of $\PP^1$.
%So these families yield only two new
% Shimura subvarieties generically contained in $j(\Mg)$.
In \cite{fpp}
also families of Galois covers over curves of genus $g'>1$ were
considered. No example was found, but it was shown that when
\eqref{star} holds then $g \leq 6 g' +1 $.

Our first result in this paper is the following:

\begin{teo}
  \label{main1}
  The only positive dimensional families of Galois covers
  $C \ra C'=C/G$ with $g' \geq 1$ and satisfying \eqref{star} are the
  6 families found in \cite{fpp}. In particular all of them have
  $g'=1$.
\end{teo}
See Theorem \ref{bau3}.  This shows that condition \eqref{star} is
very strong when $g' >0$.  Of course the moduli image of some family
could be a Shimura subvariety even if \eqref{star} does not hold.

To prove Theorem \ref{main1} we first prove that $g' \leq 3$ (Theorem
\ref{bau1}). This reduces the problem to the analysis of a finite
number of cases.  The \'etale covers are ruled out using some
elementary representation theory (Lemma \ref{etale}).  The ramified
cases are checked by a computer program as in \cite{fpp}.

One of the 2 families of coverings over elliptic curves found in
\cite{fpp} had been studied previously by Pirola \cite{pirola-Xiao} in
order to disprove a conjecture of Xiao.  This family was also studied
by Grushevsky and M\"oller \cite{gm2}, who got the following
remarkable result: the Prym map for this family is a fibration in
curves, which are totally geodesic. As consequence they obtained
uncountably many totally geodesic curves generically contained in
$\T_4$, countably many of which are Shimura.

In this paper we show that this phenomenon appears for all the Shimura
subvarieties found in \cite{fpp}:
\begin{teo}
  \label{main2}
  Consider a positive dimensional family of Galois covers $C \ra C'=C/G$ with $g' \geq 1$
  and satisfying \eqref{star} (i.e. one of the 6 families in
  \cite{fpp}). If $F$ is an irreducible component of a fibre of the
  Prym map, then $W:= \overline{j(F)}$ is a totally geodesic
  subvariety of $\Ag$ of dimension $g' (g'+1) /2$.
\end{teo}
(See Theorem \ref{bau} for the precise statement.)  In particular also
$\T_2$ and $\T_3$ contain uncountably many totally geodesic curves and
countably many Shimura curves.

Next we study the map $\phi$ that maps the covering $[C\ra C' ] $ to
$ [JC']$. Then $\phi$ is a fibration of these families onto $\A_{g'}$.

\begin{teo}
  \label{main3}
  Consider a  family of Galois covers $C \ra C'=C/G$ with $g' \geq 1$
of dimension $N>0$  and satisfying \eqref{star} (i.e. one of the 6 families in
  \cite{fpp}). If $F$ is an irreducible component of a fibre of $\phi$,
  then $X:= \overline{j(F)}$ is a totally geodesic subvariety of $\Ag$
  of dimension $ N - g' (g'+1) /2 $.
\end{teo}

Also in this case countably many of the irreducible components of the
fibres of $\phi$ are Shimura subvarieties.

In light of Theorem \ref{main1}, Theorems \ref{main2} and \ref{main3}
concern just 6 explicit families of coverings.  Nevertheless it seems
hard to prove them by a direct analysis of the families. Our approach
is different, we prove our results using some general arguments on
isogenies and totally geodesic subvarieties that are explained in
Section 3. These are of independent interest.

In the last part of the paper we analyse several features of the
examples in relation with the two fibrations $\prym$ and $\phi$.  We
describe all the inclusions among the families of Galois covers of
$\PP^1$ or of elliptic curves yielding Shimura subvarieties of $\A_g$
known so far. We show that some of them occur as irreducible
components of fibres of the Prym map or are contained in fibres of the
map $\phi$ of one of the 6 families in \cite{fpp}.
\medskip

{\bfseries \noindent{Acknowledgements}}.  The authors would like to
thank the referee for a very careful reading of the paper.

\section{Families of coverings with $g'>0$}

For $r\geq 0$ and $g'\geq 1$ set
$\Ga_{g',r}:= \sx \alfa_1, \beta_1, \lds, \alfa_{g'}, \beta_{g'},
\ga_1, \lds, \ga_r| $
$\prod_{i=1}^r \ga_i\cd \prod_{j=1}^{g'} [\alfa_j, \beta_j] \rangle$.
A \emph{datum} is a pair $\Datum:=(G, \theta) $, where $G$ is a finite
group and $\theta: \Ga_{g',r} \ra G$ is an epimorphism. We denote by
$m_j$ the order of $\theta(\ga_j)$.  After making some choices, one
can associate to a datum $(G,\theta)$ a monomorphism from $G$ to the
mapping class group $ \Map_g$, where $g$ depends on $g'$ and
$m_1, \lds, m_r$ via Riemann-Hurwitz formula.  Denote by $\teich_g$
the Teichm\"uller space and by $\teich^G_{g}$ the set of fixed points
of $G$, which is a non-empty connected complex submanifold of
$\teich_g$ of dimension $3g' -3 + r$.  The image of $\teich_g^G$ in
$\M_g$, which we denote by $\M_\Datum$, depends only on the datum, not
on the choices made. It is an algebraic subvariety of dimension
$3g' -3 + r$.  As explained in \cite[p. 79]{gab} there is an
  intermediate variety $\tmd$ such that the projection factors through
  \begin{gather}
    \label{norma}
    \teich^G_g \lra \tmd \stackrel{\nu}{\lra} \md.
  \end{gather}
  The variety $\tmd $ is the normalization of $\md$ and it
  parametrizes equivalence classes of curves with an action of $G$ of
  topological type $\Datum$.  (We denote points in $\tmd$ by
  $[C \ra C'=C/G]$ or simply by $[C]$.)  The equivalence class of the
  representation of $G$ on $H^0(C,K_C)$ does not change for $[C] $
  varying in $ \tmd$. Therefore the number
\begin{gather*}
  N (\Datum): = \dim \left ( S^2H^0(C,K_{C})\right )^G
\end{gather*}
depends only on the datum.  See \cite[\S 2]{fpp} and references
therein for more details on this construction.
  
\begin{teo}
  \label{criterio0}
  Fix a datum $\Datum=\datum$ and assume that
  \begin{gather}
    \label{bona}
    \tag{$\ast$} N (\Datum) = 3g'-3 +r.
  \end{gather}
  Then $\overline{j (\M_\Datum)} $ (closure in $\A_g$) is a special
  subvariety of PEL type of $\ag$ that is generically contained in the
  Torelli locus.
\end{teo}
(See \cite[Thm. 3.9]{fgp} and \cite[Thm. 3.7]{fpp}.)  In principle
condition \eqref{bona} is only sufficient for $\overline{j(M_\Datum)}$
to be special. It is known to be necessary if $G$ is cyclic
\cite{moonen-special} or if $G$ is abelian, $g'=0$ and $r=4$
\cite[Thm. 6.2]{mohajer-zuo-paa}.  It is not known whether it is
necessary in the general case.
  
In the following we concentrate on the case where $g'\geq 1$.  A first
result was proved in \cite[Thm. 4.11]{fpp}.
\begin{teo}
%  \label{xiao}
  If $\Datum=\datum$ is a datum with $g'\geq 1$, $3g' + r > 3 $
  (i.e. $\dim \M_\Datum > 0 $) and which satisifies condition
  \eqref{bona}, then $g \leq 6g'+1$.  In particular, for $g \geq 8 $,
  (resp. $g\geq 14$), there is no positive-dimensional datum with
  $g' =1$, (resp. $g'=2$) and which satisfy condition $(*)$.
\end{teo}

Fix a datum $\Datum=\datum$ and a point of $ \tmd$. This point
represents the isomorphism class of a curve $C$ of genus $g$, which
admits an effective holomorphic action of $G$. Denote by $C':= C/G$
the quotient, which has genus $g'$ and by $f: C \ra C'$ the quotient
map. Then $f$ is a Galois covering with branch locus $B$.  The
unramified covering $f\restr{f\meno(C'-B)}$ is the $G$-cover
associated to the epimorphism
$\theta: \Ga_{g',r}\cong \pi_1(C'-B) \ra G$.  If $g'\geq 1$, we can
consider the norm map $\Nm : JC \ra JC'$ and the Prym variety
$P (f) := (\ker \Nm)^0$ of the covering $f: C \ra C'$.  Call $\delta$
the type of the polarization obtained by restricting the theta divisor
of $JC$ to $P(f)$.  Denote by $\A^{\delta}_{g-g'}$ the moduli space of
abelian varieties of dimension $g-g'$ with a polarisation of type
$\delta$.  Since $f$ and $\Nm$ are $G$-invariant, $P(f)$ is
$G$-invariant, thus $G$ is a group of automorphisms of $P(f)$ as a
polarized abelian variety.  Denote by
\begin{equation*}
  \ada  \subset \mathsf{A}^\delta_{g-g'}
\end{equation*}
the Shimura variety parametrizing abelian varieties with an action of
$G$ of the same type as $P(f)$. This variety is constructed as in
\cite[\S 3]{fgp}.  Next denote by
\begin{gather*}
  \prym: \tmd \ra \ada
\end{gather*}
the Prym map that associates to $[C] \in \tmd$ the isomorphism class
of $(P(f),$ $ \Theta\restr{P(f)})$.  If $g'=0$, then $\prym$ is just
the Torelli morphism composed with $\nu$ in \eqref{norma}, so we are
just in the setting of Theorem \ref{criterio0}, which in fact asserts
that under condition (\ref{bona}) we have $\overline{j(\md) } =
\ada$. Instead, when $g'\geq 1$, the Prym map gives rise to some
additional geometry.
\begin{teo}
  \label{bau1}
  Consider a datum $\Datum=\datum$ with $g'\geq 1$, $3g' + r > 3 $
  (i.e. $\dim\md >0$), and which satisifies condition
  \eqref{bona}. Then $g' \leq 3$ and $\prym$ is dominant.
\end{teo}
\begin{proof}
  Both $\tmd$ and $\ada$ are complex orbifolds and $\prym$ is an
  orbifold map.  We wish to show that $\prym: \tmd \ra \ada$ is
  generically submersive.

  Fix a point $x=[C]\in \tmd$, denote by $f : C \ra C':=C/G$ the
  covering and set $(A,\Theta)=P(f)$, so that
  $ \prym(x) = [A,\Theta]$.  The orbifold tangent space of $\tmd$ at
  $x$ is $H^1(C,T_C)^G$, while the orbifold tangent space of $\ada$ at
  $\prym(x) $ is $(S^2H^0(A,\Omega^1_A)^*)^G$.  Observe that
  \begin{equation}
    \label{h0}
    H^0(K_C)= H^0(K_C)^G \oplus H^0(K_C)^-
  \end{equation}
  where $H^0(K_C)^G \cong H^0(C',K_{C'})$ and
  $ H^0(K_C)^- \cong H^{0}(A,\Omega_A^1)$ denotes the sum of the
  non-trivial isotypic components as a representation of $G$.
  Moreover $T_x^* \tmd \cong H^0(C,2K_C)^G$ and
  $T_{\prym(x)}^* \ada \cong \left ( S^2H^0(C,K_C)^- \right)^G $. So
  the codifferential (i.e. the dual of the differential) of $\prym$ at
  $x$ is a map
  \begin{gather*}
    d\prym_{x}^*: (S^2H^{0}(K_C)^-)^G \ra H^0(C, 2K_C)^G.
  \end{gather*}
  This map is just the restriction of the multiplication map
  \begin{gather}
    \label{mmult}
    m : S^2H^{0}(K_C) \ra H^0(C, 2K_C).
  \end{gather}
  Denote by $\HE_g \subset \M_g$ the hyperelliptic locus.  Assume
  first that $\M_\Datum$ is not contained in $\HE_g$ and that
  $x \not \in \HE_g$. Then the multiplication map \eqref{mmult} is
  surjective by Noether theorem. By Schur lemma
  $m ( (S^2H^{0}(K_C))^G ) = H^0(C, 2K_C)^G$.  Hence condition
  \eqref{bona} implies that
  \begin{gather}
    \label{riemme}
    m\restr{(S^2H^{0}(K_C))^G} : (S^2H^{0}(K_C))^G\lra H^0(C, 2K_C)^G
  \end{gather}
  is in fact an isomorphism.  But
  $(S^2H^{0}(K_C)^-)^G \subset (S^2H^{0}(K_C))^G$, so we conclude that
  $d\prym_{x}^*$ is injective. Hence $d\prym_{x}$ is surjective. This
  shows that $\prym$ is submersive at $x$.

  Assume now that $ \M_\Datum \subset \HE_g$ and denote by
  $\sigma: C \ra C$ the hyperelliptic involution.  Then $m$ maps
  $S^2H^{0}(K_C)$ onto $ H^0(C, 2K_C)^{\langle \sigma \rangle}$.  If
  $\sigma \in G$, then
  $H^0(C,2K_C)^G \subset H^0(C, 2K_C)^{\langle \sigma \rangle}$.  Just
  as before Schur lemma shows that \eqref{riemme} is onto
  % $m$ maps $ (S^2H^{0}(K_C))^G $ onto $ H^0(C, 2K_C)^G $.
  and \eqref{bona} yields that \eqref{riemme} is an isomorphism. It
  follows that $\prym$ is submersive at $x$.  If instead
  $\sigma \not \in G$, denote by $\tilde{G}$ the subgroup of $\Aut(C)$
  generated by $G$ and $\sigma$.  Arguing as above we conclude that
  the multiplication map
  $(S^2H^{0}(K_C))^{\tilde{G}} \ra H^0(C, 2K_C)^{\tilde{G}} $ is
  surjective. Since $ \M_\Datum \subset HE_g$, we have
  $H^0(2K_C)^G = H^0(2K_C)^{\tilde{G}}$ and since $\sigma$ acts by
  multiplication by $-1$ on $H^0(K_C)$, it acts trivially on
  $S^2H^0(K_C)$, therefore
  $(S^2H^{0}(K_C))^G = (S^2H^{0}(K_C))^{\tilde{G}}$. So also in this
  case the multiplication map \eqref{riemme} is surjective, by
  \eqref{bona} it is an isomorphism and $d\prym_{x}$ is surjective.

  We have proved that in case $ \M_\Datum \subset \HE_g$, $\prym$ is
  submersive on $\tmd $ in the orbifold sense, while in case
  $ \M_\Datum$ is not contained in $\HE_g$, $\prym$ is submersive on
  $\nu\meno(\M_\Datum - \HE_g)$.  At any case $\prym$ is
  generically submersive (in the orbifold sense) hence it is dominant.

  From \eqref{h0} we get
  \begin{equation}
    \label{W}
    \begin{gathered}
      (S^2H^0(K_C))^G
      % = S^2 ( H^0(K_C)^G) \oplus (S^2 H^0(K_C)^-)^G \oplus\\
      % \oplus
      % \left ( H^0(K_C)^G \otimes H^0(K_C)^-\right )^G \\
      \cong S^2H^0(K_{C'}) \oplus (S^2H^{0}(K_C)^-)^G .
    \end{gathered}
  \end{equation}
  Since the multiplication map \eqref{riemme} at a generic point is an
  isomorphism, its restriction to $S^2H^0(K_{C'}) $ is
  injective. Moreover it maps $S^2H^0(K_{C'}) $ to
  $H^0(2K_{C'}) \subset H^0(2K_C)^G$. Hence
  $\dim(S^2H^0(K_{C'}) ) \leq \dim H^0(2K_{C'})$, which yields
  $g' \leq 3$.
\end{proof}

\begin{lemma}
  \label{etale}
  There do not exist positive dimensional families of \'etale
  coverings $f:C \ra C'=C/G$ with $g'= g(C') \geq 2$ satisfying
  condition $(*)$.
\end{lemma}
\proof Assume that a family of \'etale coverings is given satisfying
\eqref{star} and $g' \geq 2$.  Then
$\dim (S^2 H^0(K_C))^G = \dim H^0(2K_C)^G = 3g'-3$ since $r=0$. We
have $H^0(K_C) \cong H^0(K_{C'}) \oplus V^-$, where
$V^- = \oplus_{\chi \in I} \nu_{\chi} V_{\chi}$, where $I$ is the set
of non-trivial irreducible characters of $G$.  So
$(S^2H^0(K_C))^G \cong S^2H^0(K_{C'}) \oplus (S^2(V^-))^G$ and
\begin{equation*}
  \begin{gathered}
    \dim (S^2 H^0(K_C))^G = 3g'-3 = \\
    \dim S^2H^0(K_{C'}) + \dim
    (S^2(V^-))^G \geq \\
    \dim S^2H^0(K_{C'}) = \frac{g'(g'+1)}{2} = 3g'-3,
  \end{gathered}
\end{equation*}
since $g'=2,3$. This implies $(S^2(V^-))^G =0$. By Chevalley-Weil
formula \cite{chewe} (see also \cite{naeff} or
\cite[Thm. 2.8]{fragle}) we have
$\nu_{\chi} = (\dim V_{\chi})(g'-1) >0$, for all non-trivial
irreducible character ${\chi}$. So for any $\chi \in I$ we have
$(S^2V_{\chi})^G =0$ and for any $\chi, \chi' \in I$, we have
$(V_{\chi} \otimes V_{\chi'} )^G=0$ if ${\chi} \neq {\chi'}$.  If
there is a non-trivial 1-dimensional representation $V_\chi$, this is
impossible. In fact let $\chi'$ be the character of $V_\chi^*$.  If
$\chi \neq \chi'$, then $(V_{\chi} \otimes V_{\chi'})^G \neq 0$. If
${\chi} = {\chi'}$, then
$0 \neq (V_{\chi} \otimes V_{\chi})^G \cong (S^2 (V_{\chi}))^G$, since
$\Lambda^2V_{\chi} =0$ because $\dim V_{\chi}=1$.

By Theorem \eqref{bau1} we know that $g'\leq 3$ and from of
\cite[Thm. 1.2]{fpp} that $g \leq 6g'+1$.  Denote by $d:= |G|$. If
$g'=2$, we have $d = g-1 \leq 6 g' = 12$ by Riemann-Hurwitz, while in
case $g'=3$, we have $g-1 = 2d$, hence
$d = \frac{g-1}{2} \leq 3g' =9$.  So at any case $d \leq 12$ and all
groups with $|G| \leq 12$ admit non-trivial 1-dimensional irreducible
representations.  This concludes the proof.  \qed

\begin{teo}
  \label{bau3} The only positive dimensional families of Galois
  coverings $f:C \ra C'=C/G$ with $g'= g(C') \geq 1$ and $g = g(C)$
  which satisfy condition $(*)$ have $g'=1$ and are the 6 families
  found in \cite{fpp}.
\end{teo}

\begin{proof}
  From Theorem \eqref{bau1} we know that $g'\leq 3$. From Theorem 1.2
  of \cite{fpp} we know that $g \leq 6g'+1$. From Lemma \ref{etale} we
  know that the covering has to be ramified and by computer
  calculations as in \cite{fpp} we find exactly the 6 families of
  \cite{fpp} with $g'=1$.
\end{proof}

\section{Families of isogenous Abelian varieties}

% \begin{say}
\emph{Riemannian} symmetric spaces appear in Hodge theory as parameter
spaces of Abelian varieties. In the sequel we will compare totally
geodesic submanifolds of two Siegel spaces with respect to different
polarizations. To do that we will embed both Siegel spaces in a larger
parameter space for complex tori that is a \emph{non-Riemannian
  symmetric space}.  Therefore we start by recalling the definition
and the main facts regarding this more general class of symmetric
spaces, that are less known than the Riemannian ones. We will follow
\cite[vol. II, Ch. XI]{kn}.
% \end{say}

% \begin{say}
The definition of a general (i.e. not necessarily Riemannian)
symmetric space is obtained from definition of Riemann symmetric space
by dropping the request that the connection is compatible with a
Riemannian metric. More precisely, consider a differentiable manifold
$M$ with an affine connection $\nabla$, i.e. a linear connection on
$TM$. With the aid of $\nabla$ one can define geodesics, the
exponential map, completeness and curvature. A diffeomorphism
$f: M\ra M$ is an \emph{affine transformation} if it preserves the
connection \cite[vol. I, p. 225]{kn}.  The group $\Aff(M)$ of all
affine transformations is a Lie group acting smoothly on $M$
\cite[vol. I, p. 229]{kn}.  We say that $(M, \nabla)$ is a
\emph{symmetric space} if (1) $\nabla$ is symmetric
i.e. $T(\nabla)=0$, (2) $M$ is connected, (3) $M$ is complete with
respect to $\nabla$ and (4) for each point $x \in M$ there is a
\emph{symmetry} at $x$, i.e. an affine transformation $s_x : M\ra M$
such that $s_x(x) = x$ and $(ds_x) _x = -\id_{T_x M}$.  As in the
Riemannian case there is a local version of condition (4) which is
equivalent to the fact that $\nabla R = 0$.

Assume that $(M,\nabla)$ is a symmetric space and let $G $ denote the
connected component of the identity in $ \Aff (M)$. Then $G$ acts
transitively on $M$.  Fix a point $o\in M$.  If $g\in G$, then
$\sigma (g) : = s_o \circ g \circ s_o \in G$ and $\sigma : G \ra G$ is
an involutive automorphism of $G$. If we set $H:= G_o$, then $M=G/H$
and $H$ lies between $G^\sigma = \{ g \in G: \sigma(g) = g\}$ and the
identity component of $G^\sigma$.

% \begin{say}
\label{KleinSS}
Conversely assume that $(G,H,\sigma )$ is a \emph{symmetric triple},
i.e.  $G$ is a connected Lie group, $\sigma $ is an involutive
automorphism of $G$ and $H$ is a closed subgroup lying between
$G^\sigma$ and its component of the identity.  Then $M:=G/H$ is a
\emph{reductive homogeneous space} (see \cite[vol. II, ch. X]{kn}.  As
such it admits the so-called \emph{canonical connection}
$\nabla$. With this connection $M$ is a symmetric space,
\cite[vol. II, pp. 230-231]{kn}.  If
$\liem : = \{X\in \lieg: d\sigma (X) = -X\}$ then
$\lieg = \lieh \oplus \liem$ and there is an isomorphism
\begin{gather}
  \label{isoliem}
  \liem \cong T_oM, \quad X \mapsto \desudtzero \exp(tX)\cd o.
\end{gather}
% \end{say}

% \begin{say}
\label{tot}
If $M$ is a manifold with an affine connection $\nabla$ and
$M' \subset M$ is a submanifold, denote by $p$ the canonical
projection $TM\restr{M'} \ra N:= TM\restr{M'} /TM'$. The second
fundamental form is defined as $\seff(X,Y):= p (\nabla_XY)$.  If
$\nabla$ is symmetric, then $\seff \in \Gamma (S^2T^*M' \otimes N)$.
A \emph{totally geodesic submanifold} of $M$ is a submanifold $M'$
such that any geodesic $\ga: \R \ra M$ passing through $x\in M'$ at
time $t=0$ with $\dot{\ga}(0) \in T_x M'$ remains in $M'$ for $|t|$
small enough.  It can be proved in the usual way that $M'$ is totally
geodesic iff its second fundamental form vanishes identically.
% \end{say}

% \begin{say}
\label{subtriple}
Let $(G,H,\sigma)$ be a symmetric triple. A \emph{subtriple} is a
triple $(G',H',\sigma')$ with $G'$ a connected Lie subgroup of $G$
invariant by $\sigma$, $H'=G'\cap H$ and
$\sigma':= \sigma \restr{G'}$.  A subtriple is automatically a
symmetric triple.
% \end{say}

\begin{teo}
  \label{totSS} (1) Let $(G,H,\sigma)$ be a symmetric triple and let
  $(G',H',\sigma')$ be a subtriple. Then the inclusion $G'\subset G$
  induces an embedding $G'/H' \cong M':= G'\cd o \subset M=G/H$ and
  $M'$ is a totally geodesic submanifold of $M$. The canonical
  connection of $M$
  restricts to the canonical connection of $M'$.\\
  (2) Conversely set $o:= H \in M= G/H$ and let $M' \subset M$ be a
  complete and connected totally geodesic submanifold of $M$.  Set
  $G''=\{ g \in G: g (M')=M'\}$. Then $G''$ is a Lie subgroup of $G$.
  Let $G'$ denote the identity component of $G''$. Then $G'$ is
  invariant by $\sigma$, so
  $(G', H':=H \cap G', \sigma':=\sigma \restr{G'})$ is a subtriple and
  $M'=G'/H'$.\\
  (3) If $M' \subset M$ is a totally geodesic submanifold through $o$,
  then via \eqref{isoliem} we have $T_oM' \cong \liem'$ for a
  \emph{Lie triple system} $\liem '$, i.e. for a vector subspace
  $\liem '\subset \liem$ satisfying
  $[[\liem', \liem'],\liem'] \subset \liem'$.
\end{teo}
See \cite[vol. II, p. 234-237]{kn} for a proof.

% \begin{say}
Set $V:=\R^{2g}$ and define
\begin{gather}
  \label{eq:1}
  \co (V):= \{J \in \End V: J^2 = -\id_V\}
\end{gather}
This set is invariant by the adjoint action of $\Gl(V)$ on $\End V$.
Let $V_J$ denote the complex vector space with underlying real space
$V$ and complex multiplication defined by $i\cd v := Jv$.  Fix
$J\in \co(V)$, and let $J'$ be another point of $\co(V)$.  Fix bases
$\{e_i\}$ and $\{e'_i\}$ of $V_J$ and $V_{J'}$ respectively.  Thus
$\{e_1, \lds, e_g, Je_1, \lds, Je_g\}$ is a basis of $V$ and the same
for $\{e'_i, J'e'_i\}$. Hence there is a unique map $a \in \Gl(V)$
such that $a(e_i) = e'_i$ and $a(Je_i) = J'e'_i$. It follows that
$aJ = J'a$, i.e. $\Ad a (J) = J'$. This shows that
$\co(V) \cong \GL(V) /\GL(V_J)$.  Thus $\co(V)$ is a manifold with two
connected components. The connected component containing $J$ is the
orbit $\Gl^+ (V) \cd J \cong \Gl^+(V) / \Gl(V_J)$.
% \end{say}

% \begin{say}
Fix $J \in \co(V)$ and consider the automorphism
\begin{gather*}
  \sigma_J : \GL(V) \ra \GL(V), \quad \sigma_J (a) : = \Ad J (a) = - J
  aJ.
\end{gather*}
Since $J^2 = -\id_V$, $\sigma_J$ is involutive and $\GL(V_J)$ is its
fixed point set.  Thus we have proved that the connected component of
$\co(V)$ containing $J$ is the symmetric space associated with the
symmetric triple $(\Gl^+(V), \Gl(V_J), \sigma _J)$.  We can replace
$\Gl^+(V)$ with $\Sl(V)$ and $\Gl(V_J)$ with $ \SL(V) \cap
\GL(V_J)$. Since $\SL(V)$ is simple, Theorem 3.4 of \cite[vol. II,
p. 232]{kn} implies that this space admits a symmetric
pseudo-Riemannian metric.  On the other hand, since $\Sl(V)$ is simple
and the stabilizer of $J$ in $\Sl(V)$, that is $\Sl(V)\cap \Gl(V_J)$,
is non-compact, this symmetric space is not Riemannian.  In the
following we will often refer to $\co(V)$ as a symmetric space, even
tough this is true only of its connected components.
% \end{say}

% \begin{say}
Let $\om$ be a symplectic form on $V$.  If $J \in \co(V)$, we have
$J^*\om = \om$ if and only if the bilinear form
$g_J:= \om (\cd, J\cd )$ is symmetric.  Set
\begin{gather*}
  \sieg(V,\om) := \{ J\in \co( V): J^*\om=\om, g_J \text{ is positive
    definite}\}.
\end{gather*}
This is the \emph{Siegel space} of $(V,\om)$.
% \end{say}
\begin{prop}
  \label{siegtot}
  The Siegel space $\siegv$ is a totally geodesic submanifold of
  $\co(V)$.  It is itself a symmetric space and in fact a Riemannian
  one.
\end{prop}
\begin{proof}
  It is well-known that $\siegv \neq \vacuo$ and that it is
  connected. Fix $J\in \siegv$. Then $\Sp(V,\om)$ is invariant by
  $\sigma_J$, since $\sigma_J = \Ad J$ and $J\in \Sp(V,\om)$.  Set
  $G':=\Sp(V,\om) $, $H':= G'\cap \Gl(V_J) $ and
  $\sigma_J' := \sigma_J \restr{G'}$. Then $(G', H', \sigma_J')$ is a
  subtriple of $(\Gl^+(V), \GL(V_J)\cap \GL^+(V), \sigma_J)$, so by
  Theorem \ref{totSS} $\siegv = G'/H'$ is totally geodesic submanifold
  and also a symmetric space itself.
  % One can verify that $\sigma'_J(A) = A^{T_J}$, where $A^{T_J}$
  % denotes transposition with respect to the scalar product $g_J$.
  On $V_J$ consider the Hermitian product
  $H_J (x,y):= g_J(x,y) - i \om (x,y)$. Then $ H'$ is the unitary
  group $ \operatorname{U}(V_J,H_J)$.  Since this is a compact group
  the symmetric space $\siegv$ is Riemannian, see
  \cite[p. 209]{helgason}.
\end{proof}

% \begin{say}
Set $\Lambda:=\Zeta^{2g}$.  As usual if $F \subset \C$ is a field we
set $\Lambda_F:= \La \otimes _\Zeta F$.  Then $V= \Lambda_ \R$ and
$T:=V/\Lambda$ is a real torus of dimension $2g$.  Since the tangent
bundle to $T$ is trivial, any $J\in \co(V)$ yields a complex structure
on $T$. We denote $T_J$ the complex torus obtained in this way. Any
complex torus of dimension $g$ is isomorphic to $T_J$ for some $J$. In
this sense $\co(V)$ is a parameter space for $g$-dimensional complex
tori. If $f: T_{J} \ra T_{J'}$ is an isomorphism, then $f$ lifts to an
isomorphism $a : V \ra V$ such that $a (\La) = \La$ and
$\Ad a( J) = J'$.  Thus $T_J$ and $T_{J'}$ are isomorphic iff there is
$a\in \Gl(\Lambda)$ such that $\Ad a (J) = J'$.

More generally an isogeny of $T_J$ onto $T_{J'}$ is a surjective
morphism $f: T_J \ra T_{J'}$ with finite kernel. This lifts to a
linear map $a : V \ra V$ of maximal rank, hence invertible, such that
$J'a = aJ $ (i.e. $f$ is holomorphic) and $a(\La) \subset \La$. It
follows that $a \in \Gl(\Lambda _\QQ)$.  Conversely, given
$a\in \Gl(\La_\QQ)$ such that $J'a = aJ $, multiplying $a$ by an
appropriate positive integer $m$, we get a linear map
$m\cd a : V\ra V$ such that $m \cd a(\La) \subset \La$. This induces a
surjective holomorphic morphism $f : T_J \ra T_{J'}$, which is an
isogeny. Thus $T_J$ is isogenous to $T_{J'}$ iff there is
$a\in \Gl(\Lambda_{\QQ})$ such that $\Ad a (J) = J'$.

% \end{say}

% \begin{say}
The map $J \mapsto V_J^{0,1}$ is a diffeomorphism of $\co(V)$ onto the
set $\Omega:=\{W\in \grass(g, V_\C): W \cap \bar{W} = \{0\} \}$, which
is an open subset of the Grassmannian in the analytic
topology. \mihi{It is not Zariski open (already in case $g=1$), since
  its complement is not analytic.} So $\co(V)$ is a complex
manifold\mihi{, but it is not quasi-projective}.

% \end{say}

\begin{lemma}
  \label{lemmetto} Let $Z_1, Z_2 \subset \co(V)$ be irreducible
  analytic subsets.  Let $\Omega$ be a non-empty open subset of $Z_1$.
  Assume that for any $J_1 \in \Omega$ there is some $J_2 \in Z_2$
  such that $T_{J_1}$ is isogenous to $T_{J_2}$. Then there is
  $a\in \Gl(\La_\QQ)$ such that $\Ad a (Z_1 ) \subset Z_2$.
\end{lemma}
\begin{proof}
  Given $a\in \Gl(\La_\QQ)$, let $\Ga_a \subset \co(V) \times \co(V)$
  denote the graph of $\Ad a$.
  % , i.e.
  % $\Ga_a = \{(J, J') \in \co(V) \times \co(V): J'=\Ad a (J) \}$.
  If $\pi_j : \co(V) \times \co(V) \ra \co(V)$ denotes projection on
  the $j$-th factor, the assumption is equivalent to saying that
  \begin{gather*}
    \Omega \subset \bigcup_{a\in \Gl(\La_\QQ)} \pi_1 ( \Ga_a \cap
    \pi_2\meno (Z_2)).
  \end{gather*}
  Indeed, if $J_1\in \Omega$, there is $J_2 \in Z_2$ such that
  $T_{J_1} $ and $T_{J_2}$ are isogenous, i.e. there is
  $a\in \Gl(\La_\QQ)$ such that $(J_1,J_2) \in \Ga_a$.

  For each $a\in \Gl(\La_\QQ)$ the intersection
  $ \Ga_a \cap \pi_2\meno (Z_2)$ is an analytic subset, so we can find
  a sequence $\{K_{a,i}\}_{i\in \N}$ of compact subsets of
  $\co(V)\times \co(V)$ such that
  \begin{gather*}
    \Ga_a \cap \pi_2\meno (Z_2) = \bigcup_{i=1}^\infty K_{a,i}.
  \end{gather*}
  Then
  \begin{gather*}
    \Omega = \bigcup_{a,i} \Omega \cap \pi_1 (K_{a,i}).
  \end{gather*}
  Since $K_{a,i}$ is compact, the set $\Omega \cap \pi_1 (K_{a,i})$ is
  closed in $\Omega$. As $i$ and $a$ vary in countable sets, Baire
  theorem \cite[p. 57]{bredon} implies that there are $i$ and $a$ such
  that $ \pi_1 (K_{a,i})$ contains an open subset $U$ of $\Omega$. The
  set $U$ is clearly open also in $Z_1$ and satisfies
  $U \subset \pi_1 (\Ga_a\cap \pi_2\meno(Z_2))$.  This means that if
  $J\in U$, there is $J'\in Z_2$ such that $(J,J') \in \Ga_a$. In
  other words $\Ad a (J) \in Z_2$ for any $J\in U$. So, setting
  $f=\Ad a: \co (V) \ra \co(V)$, we have $f(U)\subset Z_2$.  Therefore
  $ f\meno(Z_2)\cap Z_1$ is an analytic subset of $Z_1$, which
  contains the open subset $U \subset Z_1$.  By the Identity Lemma
  \cite[p. 167]{grauert-remmert-cas} this implies that
  $ f\meno(Z_2)\cap Z_1 = Z_1$ i.e.  $f(Z_1) \subset Z_2$.
\end{proof}

% Recall the following observation.
% \begin{lemma}
%   \label{totintot}
%   Let $(M,g)$ be a Riemannian manifold and let $M' \subset M$ be a
%   totally geodesic submanifold of $M$. If $M''$ is a submanifold of
%   $M'$, then $M'$ is totally geodesic in $M$ if and only if it is
%   totally geodesic in $M'$.
% \end{lemma}

\begin{prop}
  \label{puff}
  Let $\om_1, \om_2$ be symplectic forms on $V$.  Assume that $Z_1$ is
  an irreducible analytic subset of $\sieg(V,\om_1)$ and that $Z_2$ is
  a totally geodesic submanifold of $\sieg(V,\om_2)$.  Let $\Omega$ be
  a non-empty open subset of $Z_1$ with the property that for any
  $J_1 \in \Omega$ there is some $J_2 \in Z_2$ such that $T_{J_1}$ is
  isogenous to $T_{J_2}$.  Assume moreover that $\dim Z_1 = \dim
  Z_2$. Then there is $a\in \Gl(\La_\QQ)$ such that
  $\Ad a (Z_1 ) = Z_2$. Moreover $Z_1$ is a totally geodesic
  submanifold of $\sieg(V,\om_1)$.
\end{prop}
\begin{proof}
  By Lemma \ref{lemmetto} there is $a\in \GL(\La_\QQ)$ such that
  $\Ad a (Z_1) \subset Z_2$. Since $Z_2$ is irreducible, a proper
  analytic subset of $Z_2$ is nowhere dense in $Z_2$, see
  e.g. \cite[p. 168]{grauert-remmert-cas}.  Since
  $\dim Z_1= \dim Z_2 = n$ we conclude that $\Ad a (Z_1) = Z_2$.  This
  proves the first assertion. By assumption $Z_2$ is totally geodesic
  in $\sieg(V,\om_2)$.  By Proposition \ref{siegtot} $\sieg(V,\om_2)$
  is itself totally geodesic in $\co(V)$. Thus $Z_2$ is totally
  geodesic in $\co(V)$.  The same is true for $Z_1 $ since $\Ad a$ is
  an affine transformation of $\co(V)$. As
  $Z_1 \subset \sieg(V,\om_1)$ by assumption, $Z_1$ is in fact a
  totally geodesic submanifold of $\sieg(V,\om_1)$ as desired.
\end{proof}

The following definition goes back to Moonen
\cite{moonen-linearity-1}.
  
\begin{defin}
  \label{defE}
  Let $\om$ be a polarization of type $D$.  Denote by
  $\pi: \siegv \ra \A_g^D $ the canonical projection.  A \emph{totally
    geodesic subvariety} of $\A_g^D$ is a closed algebraic subvariety
  $W \subset \A_g^D$, such that $W = \pi (Z)$ for some totally geodesic
  submanifold $Z \subset \siegv$.
\end{defin}

We wish to prove an analogue of Proposition \ref{puff} for
subvarieties of $\A_g$ instead of $\sieg_g$.  A difficulty in passing
from $\sieg_g $ to $\A_g$ comes from the fact that the map
$\pi : \sieg_g \ra \A_g^D$ is of infinite degree and ramified.  It helps
to factor $\pi$ as an unramified covering of infinite degree followed
by a finite map.  From this one easily gets Proposition \ref {corD}
below, which is enough to ``descend'' Proposition \ref{puff} to
$\A_g$.

\begin{lemma}
  \label{lemC}
  Let $X, Y, Z$ be reduced complex analytic spaces.  Let $p: X \ra Y$
  be an unramified covering and let $q : Y \ra Z$ be a finite Galois
  covering. If $Z'\subset Z$ is an irreducible analytic subset and
  $X' $ is an irreducible component of $(qp)\meno (Z')$, then
  $ q p (X') = Z'$.
\end{lemma}
\begin{proof}
  Let $\{Y_i\}$ be the irreducible components of $q\meno(Z')$. We
  claim that $q(Y_i) = Z'$ for each $i$.  Since $q$ is a finite map
  each $q(Y_i)$ is an analytic subset of $Z'$.  Obviously
  $Z'=qq\meno(Z') = \cup_i q(Y_i)$.  By Baire theorem there is some
  $i_0$ such that $q(Y_{i_0})$ has non-empty interior, therefore
  $Z'=q(Y_{i_0})$.  Since $q: Y \ra Z$ is a Galois cover with finite
  Galois group $G$, it follows that $q\meno(Z') = G\cd Y_{i_0}$, so
  for each $i$ we have $Y_i = g \cd Y_{i_0}$ for some $g\in G$ and
  hence $q(Y_i) =q(Y_{i_0}) =Z'$.  This proves the claim.  Next fix a
  point $x$ of $X'$ that does not lie in any other irreducible
  component of $(q p )\meno (Z')$.  Since
  $ p (X') \subset q\meno (Z')$ and $X'$ is irreducible, $ p (X')$ is
  contained in a unique irreducible component $Y'$ of $q\meno (Z')$.
  By the above $q(Y') = Z'$. To conclude it is enough to show that
  $ p (X') = Y'$.  Set $y:= p (x) \in Y'$. Let $u: \tilde{Y'} \ra Y'$
  be the universal cover. Fix $\tilde{y} \in u \meno(y)$. By the
  lifiting theorem there is a holomorphic map
  $\tilde{f} : (\tilde{Y'}, \tilde{y}) \lra (X, x) $ such that
  $ p \tilde{f} = f : = i u $, where $i: Y' \ra Y$ denotes the
  inclusion.  Since $Y'$ is irreducible, also $\tilde{Y}'$ is
  irreducible, hence $\tilde{f}(\tilde{Y}')$ is contained in a unique
  irreducible component, which is necessarily $X'$ since
  $\tilde{f}(\tilde{y}) = x$. So $\tilde{f}(\tilde{Y}') \subset
  X'$. It follows that
  $Y' = iu (\tilde{Y}') = f (\tilde{Y}') = p \tilde{f} (\tilde{Y}')
  \subset p (X')$. On the other hand we have $ p (X') \subset Y'$ by
  construction. Hence $ p (X') = Y'$ as desired.
\end{proof}

\begin{prop}
  \label{corD}
  Let $\om$ be a symplectic form of type $D$.  Denote by
  $\pi: \sieg(V,\om) \ra \A_g^D$ the canonical projection.  If
  $W \subset \A_g^D$ is an irreducible analytic subset, then for any
  irreducible component $Z$ of $\pi\meno (W) $ we have $\pi(Z)=W$.
\end{prop}
\begin{proof}
  The polarisation $\omega$ is a non-degenerate alternating form
  $\omega: \Lambda^2(\Lambda) \ra \Zeta$ of type $D =(d_1,...,d_g)$,
  where $d_1|d_2|...|d_g$.  Let $n $ be a natural number such that
  $(d_g, n) =1$. Consider a symplectic level $n$ structure, i.e. a
  symplectic isomorphism of the set of $n$-torsion points $A[n]$ of
  $A = V/\Lambda$ with $({\mathbb Z}/n{\mathbb Z})^{2g}$.  Denote by
  $\Gamma_D(n)$ the subgroup of the automorphisms of the pair
  $(\Lambda, \omega)$ which induce the trivial action on
  $\Lambda/n\Lambda$.  If $n$ is large enough with respect to the
  polarisation $D=(d_1,...,d_g)$, then the quotient
  $ \siegv/ \Gamma_D(n)=: \A_g^{D,(n)}$ is smooth and the map
  $p: \siegv \ra \A_g^{D,(n)}$ is a topological covering
  (cf. e.g. \cite{gengo}). The map $\pi$ factors through the
  topological covering $p: \siegv \ra \A_g^{D,(n)}
  $ and a finite
  Galois covering $q: \A_g^{D,(n)} \ra \A_g^D$.
  % The covering $p$ is
  % unramified as soon as $n\geq 3$ \cite{serre-echelon}.
  The result follows from Lemma \ref{lemC}.
\end{proof}

\begin{teo}
  \label{inag}
  Let $D_1$ and $D_2$ be types of $g$-dimensional abelian varieties.
  Let $W_ 1 \subset \A_g^{D_1}$ and $W_2 \subset \A_g^{D_2}$ be a
  irreducible analytic subsets.  Assume that there is a non-empty
  subset $U $ of $W_1$ such that (1) $U$ is open in the complex
  topology, (2) any $[A_1] \in U$ is isogenous to some $[A_2]\in
  W_2$. Then $\dim W_1 \leq \dim W_2$.  Moreover if
  $\dim W_1 = \dim W_2$ and $W_2$ is a totally geodesic subvariety,
  then $W_1$ also is totally
  geodesic. % Moreover $W_1$ and $W_2$ are isogenous
\end{teo}
\begin{proof}
  Denote by $\pi_i : \sieg(V, \om_i) \ra \A_g^{D_i}$ the canonical
  projections.  Let $Z_i $ be an irreducible component of
  $\pi_i^{-1}(W_i)$. By Proposition \ref{corD} $\pi_i(Z_i) = W_i$.
  Set $\Omega:= Z_1 \cap \pi_1\meno(U)$.  Clearly for any
  $J_1 \in \Omega$ there is some $J_2 \in Z_2$ such that $T_{J_1} $
  and $T_{J_2}$ are isogenous.  By Lemma \ref{lemmetto} there is
  $a\in \Gl(\La_\QQ)$ such that $\Ad a (Z_1 ) \subset Z_2$. Hence
  $\dim W_1 = \dim Z_1 \leq \dim Z_2 = \dim W_2$. This proves the
  first assertion.  To prove the second, recall that by definition
  \ref{defE} we can assume that $Z_2 \subset \sieg(V,\om_2)$ is a
  totally geodesic submanifold.  By Proposition \ref{puff} $Z_1$ is
  also totally geodesic, hence the result.
\end{proof}

\begin{teo}
  \label{bau}
  Consider a datum $\Datum=\datum$ with $g'\geq 1$, $3g' + r > 3 $
  (i.e. $\dim\md >0$), and which satisifies condition
  \eqref{bona}. Then for every $y \in \Im \prym$ and for every
  irreducible component $F$ of $\prym\meno(y)$, the closure
  $W:= \overline{j(F)}$ is a totally geodesic subvariety of $\Ag$ of
  dimension $g' (g'+1) /2$.
\end{teo}

\proof We start by proving the dimension statement.  First we compute
the dimension of generic fibres of $\prym$.

By \eqref{bona} $\dim \M_\Datum = \dim (S^2H^0(K_C))^G$.  Fix $x$ such
that $\prym$ is submersive at $x$. Set $y=\prym(x)$. Then, using
\eqref{W} we get
\begin{equation}
  \begin{gathered}
    \dim_x \prym\meno(y) = \dim \tmd - \dim \ada
    =\\
    \dim
    (S^2H^0(K_C))^G - \dim (S^2H^{0}(K_C)^-)^G =  \\
    \dim S^2H^0(K_{C'}) = \frac{g'(g'+1)}{2} \geq 1.
    \label{dimensione}
  \end{gathered}
\end{equation}

Hence the generic fibre of $\prym$ has dimension $g'(g'+1)/2$,
therefore $\dim W = \dim F \geq g'(g'+1)/2$.

If $y=[A,\Theta] \in \Im \prym$, denote by $W' \subset \A_g^D$ the
closure of the variety parametrizing abelian varieties isomorphic to
products $A\times j(C')$, where $[C'] \in \M_{g'}$ (with $D$ denoting
the appropriate product polarization).  Observe that $W$ is an
irreducible analytic subset of $\A_g$ and $W'$ is an irreducible
analytic subset of $\A_g^D$.  Clearly the generic point of $W$ is
isogenous to some point of $W'$.

By Theorem \ref{bau1} $g'\leq 3$, so $W'$ parametrizes abelian
varieties of the form $A\times B$ with $[B]\in \A_{g'}$.  Hence $W'$
is a totally geodesic subvariety of $\A_g^D$ of dimension $g'(g'+1)/2$.
Theorem \ref{inag} implies that $\dim W \leq \dim W'$, so
$ \dim W = g'(g'+1)/2$.  This proves the dimension statement.  Now we
can apply the second part of Theorem \ref{inag} to conclude that $W$
is totally geodesic.  \qed

  \begin{remark}
    Since we know that the data satisfying the assumptions of the
    Theorem have $g'=1$, the fibres are 1-dimensional.
  \end{remark}
  
  Fix a datum $\Datum=\datum$ with $g'\geq 1$, $3g' + r > 3 $, and
  consider the map
  \begin{equation}
    \label{phi}
    \phi: \tmd \ra \A_{g'}
  \end{equation}
  which associates to $[C\ra C']$ the Jacobian $[JC']$.
  \begin{teo}
    \label{ribau}
    Consider a datum $\Datum=\datum$ with $g'\geq 1$, $3g' + r > 3 $
    (i.e. $\dim\md >0$), and which satisifies condition
    \eqref{bona}. Then for every $y \in j(\M_{g'})$ and for every
    irreducible component $Y$ of $\phi\meno(y)$, the closure
    $X:= \overline{j(Y)}$ is a totally geodesic subvariety of $\Ag$ of
    dimension $d:=N(\Datum) - g' (g'+1) /2$.
  \end{teo}
  \begin{proof} In the proof of Theorem \ref{bau1} we have shown that
    if $(*)$ holds, for the generic point $x= [C\ra C'] \in \tmd$ the
    multiplication map
    \begin{gather}
      \label{poldo}
      m:  (S^2H^0(K_C))^G    \cong S^2H^0(K_{C'}) \oplus (S^2H^{0}(K_C)^-)^G  \ra H^0(2K_C)^G
    \end{gather}
    is an isomorphism. This implies that the restriction of $m$ to
    $S^2H^0(K_{C'})$ is injective. Since this is the codifferential of
    $\phi$ at $x$, we have proved that
$$d\phi_x: H^1(T_C)^G \ra S^2H^0(K_{C'})^*$$
is surjective.  Hence the dimension of the generic fibre
$\phi^{-1}(y)$ is $N(\Datum) - \frac{g'(g'+1)}{2}$, so
$\dim(X) = \dim(Y) \geq N(\Datum) - \frac{g'(g'+1)}{2}$.

Let $y=[J(C'),\Theta] \in j(\M_{g'})$, denote by $X' \subset \A_g^D$
the closure of the variety parametrizing abelian varieties isomorphic
to products $ j(C') \times B$, where $B \in \ada$ (with $D$ denoting
the appropriate product polarization).  Since $\ada$ is a Shimura
subvariety of $\mathsf{A}^\delta_{g-g'}$, $X'$ is a totally geodesic
subvariety of $ \A_g^D$. Its dimension equal $\dim \ada$. As noted in
the proof of Theorem \ref{bau1}
$T_{\prym(x)}^* \ada \cong \left ( S^2H^0(C,K_C)^- \right)^G $.  By
\eqref{poldo} $\dim \ada = \dim \left (S^2H^0(C,K_C) \right)^G - $
$ \dim S^2H^0(K_{C'}) = N(\Datum) - \frac{g'(g'+1)}{2}$.  The
generic point of $X$ is isogenous to some point of $X'$, so by Theorem
\ref{inag} $\dim X \leq \dim X'$, therefore
$\dim X = N(\Datum) - \frac{g'(g'+1)}{2}$.  This proves the dimension
statement.  Now we can apply the second part of Theorem \ref{inag} to
conclude that $X$ is totally geodesic.
\end{proof}

  \begin{remark}
    Since we know that the data satisfying the assumptions of the
    Theorem have $g'=1$, we have in fact $d= N(\Delta) -1 = r -1 $.
  \end{remark}

  \begin{remark}
    Theorems \ref{bau} and \ref{ribau} deal with fibres of $\prym$ and
    $\phi$.  One can apply Theorem \ref{inag} also to
    $Z:=\overline{j(\M_\Datum)}$ as a whole. In this case, remembering
    Proposition \ref{puff}, one gets an isomorphism
    $a \in \Gl(\La_\QQ)$ that maps a lifting to Siegel space of $Z$  to
    an appropriate lifting of $\A_1\times \ada$. (This proves again
    that $Z$ is totally geodesic.)  Both $Z$ and $\A_1\times \ada$
    have a \enf{product structure}: $\A_1\times \ada$ has the natural
    projections $\pi_i$ on the factors, while $Z$ has the maps $\phi$
    and $\prym$.  It is natural to ask whether this
    $a \in \Gl(\Lambda_\QQ)$ can be chosen in such a way that
    $\pi_1 \circ a = \phi$ and $\pi_2 \circ a = \prym$. This means
    that (at the level of Siegel space!) $a(y)= (\phi(y), \prym(y))$.
    It seems unlikely that such a map can be gotten by methods similar
    to those of Theorem \ref{inag}.  Since we are dealing with only 6
    families an explicit analysis could in principle answer this
    question. Yet this is probably non-trivial.
\end{remark}

\section{Analysis of the families fibred in totally geodesic
  subvarieties}
\label{sec:examples}

There are exactly 6 families with $g'=1$ that satisfy $(*)$
(i.e. $N(\Delta) =r$).  The purpose of this section is to give a
detailed analysis of these families from the point of view of the
fibrations $\prym$ and $\phi$.  The 6 families are the following:
\begin{itemize}
\item[(1e)] $g = 2$, $G = \Zeta/2\Zeta$, $N=r=2$.
\item [(2e)] $g =3$, $G= \Zeta/2\Zeta$, $N=r=4$.
\item [(3e)] $g =3$, $G= \Zeta/3\Zeta$, $N=r=2$.
\item [(4e)] $g =3$, $G= \Zeta/4\Zeta$, $N=r=2$.
\item [(5e)] $g =3$, $G= Q_8$, $N=r=1$.
\item [(6e)] $g =4$, $G= \Zeta/3\Zeta$, $N=r=3$.
\end{itemize}
All these families except, (2e) and (6e), yield Shimura subvarieties
which can be obtained also using coverings of $\PP^1$.  In fact in
\cite{fpp} it was shown that:

\begin{itemize}
\item (1e) gives the same subvariety as (26) of Table 2 in \cite{fgp}
  (this was already found in \cite{moonen-oort}).
\item (3e) gives the same subvariety as (31) of Table 2 in \cite{fgp}.
\item (4e) gives the same subvariety as (32) of Table 2 in \cite{fgp}.
\item (5e) gives the same subvariety as (34)=(23)=(7) of Table 2 in
  \cite{fgp} (see also Table 1 in \cite{fgp} to see that these
  families are the same).
\item Apart from (1e), none of these families is contained in the
  hyperelliptic locus.
\end{itemize}

\begin{cor}
  Families \emph{(1e), (2e), (3e), (4e), (6e)} are fibred in totally
  geodesic curves via their Prym maps and are fibred in totally
  geodesic subvarieties of codimension 1 via the map $\phi$.
  Therefore they contain infinitely many totally geodesic subvarieties
  and countably many Shimura subvarieties.  The Prym map of family
  \emph{(5e)} is constant, its image is the square of the elliptic
  curve $y^2 = x^3 - x$, i.e. the elliptic curve with lattice
  $\Zeta + i \Zeta$.
\end{cor}
\begin{proof} This follows immediately from Theorems \ref{bau} and
  \ref{ribau}. Since all these families yield Shimura subvarieties of
  ${\mathsf A}_g$, they contain countably many CM points, hence the
  fibres of the two maps $\prym$ and $\phi$ passing through these
  points are Shimura subvarieties. Since family (5e) is one
  dimensional, it is itself a fibre of its Prym map $\prym$, which
  therefore is constant. The computation of this constant abelian
  surface is given below in the proof of Proposition \ref{sskk}, while
  analysing the family (5e)=(7)=(23)=(34).
\end{proof}

Now we will give all the possible inclusions between the families of
Galois covers of $\PP^1$ or of genus 1 curves known so far and listed
in Tables 1,2 of \cite{fgp} and in \cite{fpp}, yielding Shimura
varieties.  In this way we will also check which of these families are
contained in a fibre of the Prym map $\prym$ or of the map $\phi$ of
one of the 6 families above.

In genus 2 we have the following diagram, where the arrows denote the
inclusions. In the lowest line we have the one dimensional families,
family (1e) $= $(26) has dimension 2, while family (2) has dimension 3
and is ${\mathsf M}_2$.
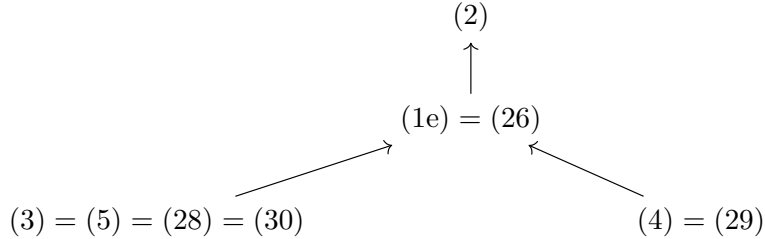
\begin{equation*}
  \begin{tikzcd}%[column sep=large]
    &       (2)  & \\
    &     (1\mathrm{e})=(26)  \arrow {u} &  \\
    (3)=(5)=(28)=(30)\arrow{ur} & & (4) = (29)\arrow {ul} &\\
  \end{tikzcd}
\end{equation*}

\begin{prop}
  In genus 2 families (3)=(5)=(28)=(30) and (4) = (29) are not
  contained in any fibre of $\prym$, nor in any fibre of $\phi$ of the
  family (1e).
\end{prop}
\proof Since both the curves in family $(30)$ and the ones in family
$(29)$ have jacobians that are isogenous to the self product of an
elliptic curve (see e.g. \cite{cfgp} p. 16), none of these two
families are fibres of the Prym map of (1e), nor of the map $\phi$ of
(1e).  \qed

Note that the bielliptic locus in genus 2 has codimension 1 in ${\mathsf M}_2$ and it is totally geodesic in ${\mathsf A}_2$. This only happens in genus 2. In fact in \cite{fgp} it is shown that if $g \geq 3$ and $Y \subset {\mathsf M}_g$ is an irreducible divisor, then there is no proper totally geodesic subvariety of ${\mathsf A}_g$ containing $j(Y)$. \\

In genus 3 we have the following diagram. In the lowest line we have
the one dimensional families, in the second one the two dimensional
ones, family (27) has dimension 3 and (2e) has dimension 4.

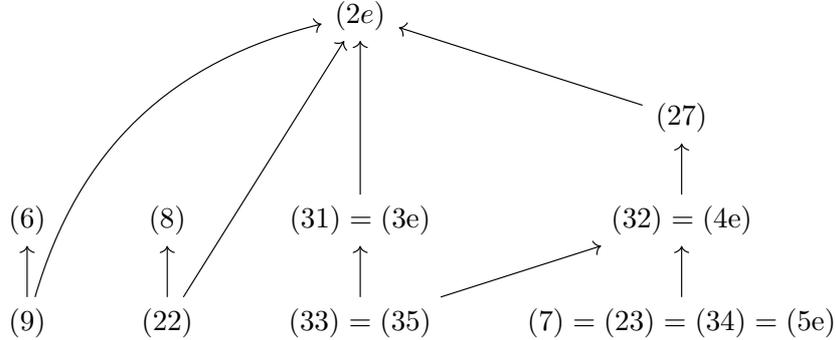
\begin{equation*}
  \begin{tikzcd}%[column sep=large]
    &    &          (2e) &     \\
    &    &       &   (27) \arrow {ul}     \\
    (6) & (8) & (31) = (3\mathrm{e})\arrow{uu} & (32) = (4\mathrm{e})  \arrow{u}   \\
    (9 ) \arrow{u} \arrow[bend left]{uuurr} & (22) \arrow{u} \arrow
    {uuur} & (33)=(35) \arrow{u} \arrow{ur} & (7) =
    (23)=(34)=(5\mathrm{e}) \arrow{u}
  \end{tikzcd}
\end{equation*}

\medskip

  \begin{prop}\label{sskk}
    \leavevmode
    \begin{enumerate}
    \item [i)] Family \emph{(34)} is a fibre of the Prym map of the
      bielliptic locus \emph{(2e)} and also a fibre of the Prym map of
      \emph{(4e)}.
    \item [ii)] Families \emph{(9)} and \emph{(22)} are both contained
      in fibres of the map $\phi$ of \emph{(2e)} and are not contained
      in any fibre of the map $\prym$ of \emph{(2e)}.
    \item [iii)] Family \emph{(33) =(35)} is not contained in any
      fibre of $\phi$ nor in any fibre of $\prym$ of
      \emph{(2e)},\emph{(3e)} or \emph{(4e)}.
    \item [iv)] Families \emph{(31) = (3e)} and \emph{(32) = (4e)} are
      not contained in fibres of the map $\phi $ of \emph{(2e)}.
    \item [v)] Family \emph{(27)} is not contained in a fibre of the
      map $\phi $ of \emph{(2e)}.
    \end{enumerate}
  \end{prop}

  \proof

  In the following, for the description of the families of \cite{fgp},
  we will denote by $\theta: \Ga_{0,r} \ra G$ the monodromy, we set
  $x_i = \theta(\gamma_i)$, $i =1,...,r$ and ${m}=(m_1, \lds, m_r)$,
  where $m_j=o(x_j)$.  For simplicity we will just write
  $x=(x_1,...,x_r)$ to describe the mondromy. We will also use the
  notation of \verb|MAGMA| for the irreducible representations of the
  group $G$.  We will analyse the different families one at a time.
  \\

\noindent
\textbf{(5e)=(7)=(23)=(34)}.\\
The description of (34) in \cite{fgp} is as follows:
\begin{gather*}
  \begin{aligned}
    G=\mathbb{Z}/4 \times\mathbb{Z}/2\rtimes\mathbb{Z}/2= \langle & g_{1},g_{2},g_{3}: g_{1}^{2}=g_{2}^{2}=g_{3}^{4}=1, \\
    &g_{2}g_{3}= g_{3}g_{2}, g_1^{-1}g_{2}{g_{1}}=g_{2}g_{3}^{2},
    g_1^{-1}g_{3}{g_{1}}=g_{3}\rangle,
  \end{aligned}
  \\
  m=(2,2,2,4), \quad x = (g_1, g_1g_2g_3^3, g_2g_3^2, g_3^3).
  % (x_1= g_1, \ x_2= g_1g_2g_3^3, \ x_3 = g_2g_3^2, \ x_4 =g_3^3).
\end{gather*}
\noindent
Using the notation of \verb|MAGMA|, we have
$H^{0}(C, K_{C})\cong V_{6}\oplus V_{10}$, where $V_i$ are irreducible
representations of $G$ and $\dim(V_6)=1$, $\dim V_{10} = 2$. Moreover
$(S^{2}H^{0}(C, K_{C}))^{G}\cong S^{2}V_{6}$, which has dimension 1,
hence condition $(*)$ holds.  The group algebra decomposition for the
curves $C$ in the family (see \cite{cr}, \cite{rojas}), gives us a
decomposition of the Jacobian $JC$ up to isogeny:
$JC\sim B_{6}\times B_{10}^{2}$, where both the $B_{i}$'s have
dimension one.
\\
\noindent
Choose $H=\langle g_{3}^{2}\rangle$ and consider the map
$f: C \ra C/H$.  Call $z_1,...,z_4$ the critical values of the map
$\psi: C \ra C/G$. One immediately verifies that there are only four
critical points of index two for $f$, all placed in the fibre
$\psi^{-1}(z_{4})$.  Applying Riemann-Hurwitz formula, we obtain
$g(C/H)=1$. Notice that this gives us the inclusion of (34) in (2e).
\\
Denote by $E:=C/H$, we get $B_{6}\sim E$, and
$H^{0}(E, K_{E})=H^{0}(C, K_{C})^{H}=V_{6}$. So the curve $E$ moves,
and the Prym variety $P(C,E) \sim B_{10}^{2}$.  So the Prym variety
$P(C,E)$ doesn't move, hence family (34) is a fibre of the Prym map of
the family (2e).
\\
Now we consider the normal subgroup
$Q_{8}=\langle g_{3}^{2},g_{2}g_{3}, g_{1}g_{3}, g_{1}g_{2}g_{3}^{2}
\rangle$ which is isomorphic to the quaternion group.  The degree 8
map $C\rightarrow C/Q_{8}$ has a single branch point and the quotient
$C/Q_{8}$ has genus 1. So (34)= (5e).
\\
Now we show that this family is also a fibre of the Prym map of family
(4e).  In fact, consider the subgroup $H'=\langle g_{1}g_{3} \rangle $
of order 4, which contains $H$, and the map $f': C \ra C/H'$. One can
check that $f'$ has four critical points of order two and two critical
values and $C/H'$ has genus 1.  This provides the inclusion
(34)$\subset$(4e).  Moreover $H^{0}(C, K_{C})^{H'}=V_{6}$, so the
curve $E'=C/H'$ moves and it is isogenous to $E = C/H$ while $P(C,E')$
is fixed and it is isogenous to $P(C,E)$. Hence (34) is also a fibre
of the Prym map of (4e).  Finally one can see that the elliptic curve
$B_{10}$ is obtained as the quotient $C/\langle g_1 \rangle$.  The
elliptic curve $B_{10}$ is the curve with j-invariant equal to $1728$
that has Legendre equation: $y^2 = x(x^2-1)$. To check this, consider
the commutative diagram
\begin{equation*}
  \begin{tikzcd}
    & &  &  C/\langle g_1 \rangle = B_{10} \arrow{d}{\eps} \\
    C \arrow{dd}{\psi} \arrow[end anchor = west,bend left]
    {urrr}{\gamma} \arrow{rrr}{\alpha}
    \arrow[end anchor=  west]{drrr}{\phi}& & & C/\langle g_1, g_3^2 \rangle = \PP^1 \arrow{d}{h}\\
    & & & C/\langle g_1, g_3 \rangle =\PP^1 \arrow[swap,bend left, start anchor = south , end anchor = east]{dlll}{\pi}\\
    C/G=\PP^1 & & &
  \end{tikzcd}
\end{equation*}
Assume that the critical values of $\psi$ are
$[z_1=\lambda, z_2 = 0, z_3 = \infty, z_4 = 1]$ with corresponding
monodromy $(g_1, g_1g_2g_3^3, g_2g_3^2, g_3^3).$ Then we can assume
that the map
$\pi: C/\langle g_1, g_3 \rangle \cong \PP^1 \ra C/G \cong \PP^1$ is
$z \mapsto z^2$. Hence the critical values of $\phi$ are
$[\mu, -\mu, 1, -1]$, where $\mu^2 = \lambda$, and corresponding
monodromy $(g_1, g_1 g_3^2, g_3, g_3)$. We can assume that
$h(z) = \frac{z^2 + c}{z^2 - c}$, where $c = \frac{1 + \mu}{\mu -1}$,
so the critical values of $\alpha$ are
$[\frac{1 + \mu}{\mu -1}, -\frac{1 + \mu}{\mu -1}, 1, -1, \infty, 0]$
with monodromy $(g_1, g_1, g_1g_3^2, g_1 g_3^2, g_3^2, g_3^2)$. So
finally we see that the critical values of the double cover
$\eps: B_{10} = C/\langle g_1 \rangle \ra C/\langle g_1, g_3^2\rangle
\cong \PP^1$ are $1, -1, \infty, 0$, thus $B_{10}$ has equation
$y^2 = x(x^2 -1)$.
\\

\noindent
\textbf{(33)=(35)}.\\
$G=S_{4}$,  set $g_{1}=(12)$, $g_{2}=(123)$, $g_{3}=(13)(24)$ and $g_{4}=(14)(23)$. Monodromy: $(x_1,x_2,x_3,x_4)= (g_{1}g_{2}^{2},g_{3}g_{4}, g_{1}, g_{2}^{2}g_{4})$, $m=(2,2,2,3)$, $H^{0}(C, K_{C})\cong V_{4}$ ($V_4$ is an irreducible representation of dimension 3), so $ (S^{2}H^{0}(C, K_{C}))^{G}\cong (S^{2}V_{4})^G $. \\
Considering the group algebra decomposition of the Jacobians of these curves we obtain $JC\sim B_{4}^{3}$, where $B_{4}$ is an elliptic curve. \\
Let's describe it. Take the subgroup $H=\langle g_{2}\rangle$, which
has order three, and consider the quotient map $f$. We have two
critical points of $f$ in $\psi^{-1}(z_{4})$ of multiplicity equal to
three. Applying Riemann-Hurwitz formula we get an elliptic curve
$E:=C/H$. We can see that $H^0(E, K_E)$ is a one dimensional subspace
of $V_{4}$, hence $E$ is isogenous to $B_{4}$. The family (35) is one
dimensional so $E$ moves, hence family (35) is not a fibre of the Prym
map of (2e) and it is not contained in a fibre of the map $\phi$ of
(2e).
Finally, observe that $f$ has exactly two critical values. This shows  the inclusion (35)$\subset$(3e). \\
% On the other hand suppose to consider $H'=<g_{1},g_{2}>$ and to study the map $\phi': C\rightarrow C/H'$. Then we have six critical points of index two in $(\phi')^{-1}(z_{1})$ which yield two branch points. The same occours in $(\phi')^{-1}(z_{3})$ while we calculate two points of order three in $(\phi')^{-1}(z_{4})$. We can use RH to get $g(C/H')=0$. Thus $\phi'$ has five branch points of signature (2,2,2,2,3). This provides us (35)$\subset$(31).\\
\\
\textbf{(9)}\\
$G:=\mathbb{Z}/6= \langle g_{1},g_{2} \ | \
g_{1}^{2}=g_{2}^{3}=1\rangle$, monodromy:
$(g_{1}, g_{2}^{2}, g_{2}^{2}, g_{1}g_{2}^{2} )$, $m=(2,3,3,6),$
$H^{0}(C, K_{C})\cong V_{4}\oplus V_{5}\oplus V_{6}$ (the $V_i$'s are
irreducible representations of $G$) and
$ (S^{2}H^{0}(C, K_{C}))^{G}\cong V_{4}\otimes V_{6}. $
\\
Let $C$ be a curve in the family and denote as usual the quotient map
by $\psi: C \ra C/G$.
Now take the subgroup $H:=\langle g_{1}\rangle$. The corresponding quotient map $C \ra C/H$  has three critical points, of index two, in $\psi^{-1}(z_{1})$ and a single one, of the same type, in $\psi^{-1}(z_{4})$. By Riemann-Hurwitz formula we see that $E:=C/H$  is an elliptic curve with $H^{0}(E, K_{E})=V_{5}.$ First observe that this curve doesn't move. Moreover we get (9)$\subset$(2e).\\
The group algebra decomposition gives $JC\sim B_{4}\times B_{5}$,
where $B_{5}$ is isogenous to the elliptic curve $E$, that is fixed,
$B_4 \sim P(C,E)$ has dimension 2 and it moves.  This shows that
family (9) is
contained in a fibre of the map $\phi$ of $(2e)$. Hence this fibre of $\phi$ has an irreducible component which is a Shimura subvariety of $\A_3$ of dimension 3. \\
\\
\textbf{(22)}\\
$ G:=\mathbb{Z}/2 \times\mathbb{Z}/4$, the monodromy is generated by
$(g_{3}, g_{2}g_{3}, g_{1}g_{2}, g_{1}g_{3} )$, where $g_{1}=(0,1)$,
$g_{2}=(1,0)$ and $g_{3}=(0,2)$, $m=(2,2,4,4)$. We have
$H^{0}(C, K_{C})\cong V_{3}\oplus V_{7}\oplus V_{8}$ and
$ (S^{2}H^{0}(C, K_{C}))^{G}\cong V_{3}\otimes V_{7}. $ The group
algebra decomposition gives $JC\sim B_{3}\times B_{8}$, with
dim$(B_{3})=2$ and dim$(B_{8})=1$. Consider the subgroup
$H=\langle g_{2}g_{3}\rangle $. One easily checks that the quotient
curve $E=C/H$ has genus one, with $(H^{0}(C,
K_{C}))^{H}=V_{8}$. Therefore $E$ is isogenous to $B_8$ and it remains
fixed.  This proves that family (22) is contained in a fibre of the
map $\phi$ of $(2e)$ This also implies that this fibre has an
irreducible component that is a Shimura subvariety of $\A_3$ of
dimension 3.

We add some details on this family which show that the jacobians of
this family are decomposable as products of elliptic curves.  Take
another subgroup: $H'=\langle g_{2}\rangle$. The quotient map
$C\rightarrow C/H'$ is \'etale so, applying Riemann-Hurwitz, we have
$g(C/H')=2$. As dim($V_{3}^{H'})=1=s_{V_3}$, where $s_{V_3}$ is the
Schur index of $V_3$, and since dim($V_{8}^{H'}$)=0 we apply
\cite[Lemma 1]{j} and we obtain $B_3\sim JC'$.

Set $C/H'=C'$ and consider the degree 4 map
$C'\rightarrow C'/\langle \overline{g_{1}} \rangle \cong \PP^1$, where
$\langle \overline{g_1}\rangle \cong G/H' \cong {\mathbb Z}/4{\mathbb
  Z}$. One easily sees that the family
$C' \rightarrow C'/\langle \overline{g_{1}} \rangle $ coincides with
the family (4) = (29). The group algebra decomposition of the
Jacobians for (29) is $JC'\sim F^{2}$, where $F$ is an elliptic
curve.  Thus we conclude $JC\sim E\times F^{2}$.  \\

Now we analize the families of dimension two ($ N=2 $).\\
\\
\textbf{(31)}\\
$G:=S_{3}=\langle g_{1},g_{2}: g_{1}^{2}=g_{2}^{3}=1,
g_1^{-1}g_{2}g_{1}=g_{2}^{2}\rangle$, with monodromy given by
$(g_{1}g_{2}^{2}, g_{1}g_{2}, g_{1},g_{1}g_{2}^{2}, g_{2}^{2} )$,
$m=(2,2,2,2,3)$. We have $H^{0}(C, K_{C})=V_{2}\oplus V_{3}$, where
$V_2$ has dimension 1, $V_3$ has dimension 2, and
$ (S^{2}H^{0}(C, K_{C}))^{G}=S^{2}V_{2}\oplus(S^{2}V_{3})^{G}. $ The
group algebra decomposition gives $JC\sim B_{2}\times B_{3}^{2}$,
where both terms have dimension one. Consider a curve $C$ of the
family and denote as usual by $\psi: C \ra C/G \cong \PP^1$ the
quotient map.  Denote by $H=\langle g_{2} \rangle \subset G$ and by
$\alpha: C \ra C/H$ the quotient map. We have two critical points of
index three in $\psi^{-1}(z_{5})$, hence $E:=C/\langle g_{2} \rangle$
has genus one and one can show that $H^{0}(E, K_{E})=V_{2}$. Thus
$B_{2}\sim E$ and we also have shown that (31)=(3e).  Finally one can
easily see that $C/\langle g_1 \rangle$ has genus 1 and
$B_{3}\sim J(C/\langle g_{1}\rangle)$. This gives the inclusion
$(31) \subset$ (2e). Notice that both $B_2$ and $B_3$ move, so $(31)$
is not contained in a fibre of the map $\phi$.\\

% Thus we can say that the curve Secondly  we control $B_{3}$. Take the subgroup $H'=<g_{1}>$ and look at the quotient map $\phi$: we get one critical point in $\phi^{-1}(z_{1})$ and the same on $\phi^{-1}(z_{2})$, $\phi^{-1}(z_{3})$ and $\phi^{-1}(z_{4})$ (we get four branch points of order two and thus the inclusion (31)$\subset $(2e)). Using RH we obtain $g(C/<g_{1}>)=1$ and $H^{0}(C, K_{C})^{<g_{1}>}=V_{3}^{<g_{1}>}$, where with the notation $V_{3}^{<g_{1}>}$ we refer to a one dimensional subspace of $V_{3}$ left invariant by the action of $<g_{1}>$. Therefore $B_{3}\sim J(C/<g_{1}>)$. \\
% Notice that both the $ B_{i} $s move. \\
% \\

\noindent
\textbf{(32)}\\
$G=D_{4}=\langle x,y: x^{4}=y^{2}=1, y^{-1}xy=x^{3}\rangle$, monodromy given by $(x^3y,y, x^2, y, xy )$,  $m=(2,2,2,2,2)$.  Moreover $H^{0}(C, K_{C})\cong V_{4}\oplus V_{5}$, where $V_4$ has dimension 1, $V_5$ has dimension 2 and $ (S^{2}H^{0}(C, K_{C}))^{G} \cong S^{2}V_{4}\oplus(S^{2}V_{5})^{G} $. The group algebra decomposition yields $JC\sim B_{4}\times B_{5}^{2}$.\\
Take the subgroup $H=\langle x \rangle$ and consider the quotient map
$ \alpha: C\rightarrow C/H$. We get four critical points in
$\psi^{-1}(z_{3})$ of index two, hence $g(C/H)=1$. This shows the
inclusion (32)$\subset$(4e).
Consider the subgroup $H'=\langle x^{2} \rangle$ of $H$. We can factor the degree four map $\alpha$ into two maps of degree two. The map $\alpha': C\rightarrow C/H'$ has four critical points of index two in $\psi^{-1}(z_{3})$, hence $ C/H'$ has genus 1 and is isogenous to $C/H$. This gives the inclusion (32)$\subset$(2e). We have $H^{1,0}(C/H)  \cong V_{4}$, therefore $C/H\sim C/H' \sim B_{4}$.\\
Consider the subgroup $K=\langle y \rangle$. One immediately checks
that $C/K$ has genus 1 and $H^{1,0}(C/K)\subseteq V_{5}$, hence
$B_{5}\sim C/K$. So both $B_4$ and $B_5$ move, hence $(32)$ is not
contained in a fibre of the map $\phi$.\\

Let's now describe the only family of dimension $N=3$. \\

\noindent
\textbf{(27)}\\
$G=\mathbb{Z}/2\times\mathbb{Z}/2=\langle g_{1},g_{2}: g_{1}^{2}=g_{2}^{2}=1 \rangle $, monodromy: $(g_{2},g_{1}g_{2}, g_{1}, g_{1}g_{2}, g_{1},g_{2} )$, $H^{0}(C, K_{C})\cong V_{2}\oplus V_{3}\oplus V_{4}$, $ (S^{2}H^{0}(C, K_{C}))^{G} \cong S^{2}V_{2}\oplus S^{2}V_{3}\oplus S^{2}V_{4}. $ The Jacobian decomposes up to isogeny as $JC\sim B_{2}\times B_{3}\times B_{4}$, where $B_i$'s are three different elliptic curves. \\
One easily checks that $B_{2}\sim C/H$, $B_{3}\sim C/H'$ and finally $B_{4}\sim C/H''$, where  $H=\langle g_{2} \rangle$, $H'=\langle g_{1} \rangle$ and $H''=\langle g_{1}g_2 \rangle$. This gives the inclusion $(27) \subset$ (2e). All the three elliptic curves move, so $(27)$ is not a fibre of the map $\phi$.\\
\qed \\

This ends up the discussion of inclusions in genus 3.  In genus 4 we
have the following diagram of inclusions.
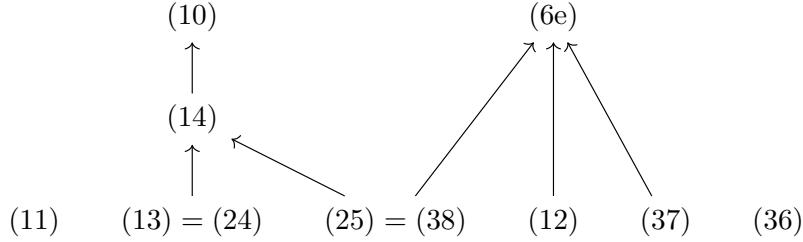
\begin{equation*}
  \begin{tikzcd}[column sep=small]
    &      (10) &  & (6\mathrm{e}) &  & \\
    &      (14) \arrow {u} & & & &    \\
    (11) & (13)=(24) \arrow {u} & (25)=(38) \arrow {ul} \arrow{uur} &
    (12) \arrow{uu} & (37) \arrow{uul} & (36)
  \end{tikzcd}
\end{equation*}
Families in the lowest line are one-dimensional, (14) has dimension 2,
while (10) and (6e) have dimension 3.

\begin{prop}
  Family (12) is contained in a fibre of the Prym map of (6e), while
  family $(25) = (38)$ is contained in a fibre of the map $\phi$ of
  (6e). Family (37) is not contained in any fibre of the Prym map nor
  of the map $\phi$ of (6e).
\end{prop}
\proof Let us start by considering family \textbf{(12)}:
\begin{gather*}
  G=\mathbb{Z}/6=\langle g_{1}g_{2}: g_{1}^{2}=g_{2}^{3}=1 \rangle, x=(g_{1},g_{1}g_{2},g_{1}g_{2},g_{1}g_{2}),\\
  m = (2,6,6,6), \quad H^{0}(C, K_{C}) \cong V_{2}\oplus V_{3}\oplus
  2V_{6}, \quad (S^{2}H^{0}(C, K_{C}))^{G} \cong S^{2}V_{2}.
\end{gather*}

\noindent
The group algebra decomposition gives
$JC\sim B_{2}\times B_{3}\times B_{6}$, where the first two terms have
dimension one while the third one has dimension equal to two.
Consider the subgroup $H:=\langle g_{2}\rangle\cong \mathbb{Z}/3$. Call $z_{1},z_{2},z_{3},z_{4}$ the branch points for the map $\psi: C \rightarrow C/G$ and consider the quotient map $f:  C \rightarrow C/H$.  There are  three critical points for $f$, of order three, in the preimage $\psi^{-1}(z_{i}), i=2,3,4$. Applying Riemann-Hurwitz formula we get that the genus of $E:=C/H$ is one. This gives us the inclusion (12)$\subset$(6e). Moreover, since  $H^{0}(C, K_{C})^{H}=V_{2}$, we get $B_{2}=E$.\\
Consider $H':=\langle g_{1}\rangle\cong \mathbb{Z}/2$. Each critical point for $\psi$ is also critical of order two for the quotient map $ C \rightarrow C/H'$. Riemann-Hurwitz formula implies that $C/H'$ is an elliptic curve. We have $H^{0}(C/H', K_{C/H'})=V_{3}$ and thus $B_{3} \sim C/H'$.\\
Notice that the terms $B_{3}$ and $ B_{6} $ don't move and their product is isogenous to the Prym variety $P(C,E)$. Thus, as already observed in \cite{gm2}, (12) is contained in a fibre of the Prym map of (6e).\\

\noindent
\textbf{(25)=(38)}\\
$G=\mathbb{Z}/3\times S_{3}$ with generators $g_{1}=([0]_3,(12))$,
$g_{2}=(1,(1))$ and $g_{3}=([0]_3,(123))$ and monodromy
$(g_{1}g_{3}^{2},g_{1}g_{3},g_{2}g_{3},g_{2}^{2})$, $m=(2,2,3,3)$. We
know that $H^{0}(C, K_{C})\cong V_{3}\oplus V_{4}\oplus V_{8}$ and
that $ (S^{2}H^{0}(C, K_{C}))^{G} \cong V_{3}\otimes V_{4}.$ The first
two $V_i$'s have dimension one while $V_{8}$ has dimension equal to
two.
The Jacobian decomposes as $JC\sim B_{3}\times B_{8}^{2}$, the first term is 2-dimensional while the second is 1-dimensional. \\
Set $H:=\langle g_{2}g_{3}\rangle \cong {\mathbb Z}/3$ and consider the quotient map $f: C\rightarrow C/H$.  We get three critical points of $f$ of order three all contained in $\psi^{-1}(z_{3})$. Hence  $g(C/H)=1$ and we also see the inclusion (38)$\subset$(6e). Moreover $H$ fixes a one-dimensional subspace of $V_{8}$, thus we get $C/H\sim B_{8}$. Note that this term of the decomposition doesn't move. This implies that family (38) is contained in a fibre of $\phi$ of (6e). Thus this fibre determines a Shimura subvariety of $\A_4$ of dimension 2.\\
Now take the quotient for $H'=\langle g_{3}\rangle \cong {\mathbb Z}/3$. The correspondent quotient map is \'etale. Therefore we obtain $g(C/H')=2$. Due to the fact that dim($V_{3}^{H'}$)=1=$s_{V_{_{3}}}$, where $s_{V_{_{3}}}$ is the Schur index of $V_3$, and since dim($V_{8}^{H'}$)=0, there exists an isogeny between $J(C/H')$ and $B_{3}$ (see \cite[Lemma 1]{j}). Since $H'$ is normal in $G$ we can now look at the map $C/H'\rightarrow \mathbb{P}^{1}$. This is a Galois covering with Galois group $G/H'\cong\mathbb{Z}/6$ and $m=(2,2,3,3)$. Actually we have only one one-dimensional family with this datum and it corresponds to family (5) of \cite{fgp}. Using the group algebra decomposition on the Jacobian of this family we get $ B_{3}\sim J(C/H')\sim E^{2}$, where $E$ is an elliptic curve. Since $E$ moves, family (38) is not contained in  a fibre of the Prym map of (6e). \\

\noindent
\textbf{(37)}\\
$G=A_{4} $, the generators for the monodromy are
$(g_{3},g_{1}g_{3},g_{1},g_{1}g_{2}g_{3})$, where $g_{1}=(123)$, $g_{2}=(12)(34)$ and  $g_{3}=(13)(24)$, $m=(2,3,3,3)$. Moreover we have $H^{0}(C, K_{C}) \cong V_{2}\oplus V_{4}$, where $ V_{2} $ has dimension one and $ V_{4} $ has dimension equal to three and $ (S^{2}H^{0}(C, K_{C}))^{G} \cong (S^{2}V_{4})^{G}$. The Jacobian decomposes completely as $JC\sim B_{2}\times B_{4}^{3}$.\\
Take the quotient $\psi:C\rightarrow C/G $ and call, as usual, the branch points $z_{i}, \ i=1,2,3,4$. Now we consider the subgroup generated by $\langle g_{1}\rangle$. It is a cyclic group of order three. Studying the map $C\rightarrow C/\langle g_{1}\rangle:=E$ we get three critical points, of order three, respectively in the fibres $\psi(z_{i})^{-1}, i=2,3,4$ and so we have $g(E)=1$ (and also the inclusion in the family (6e)). Moreover $ \langle g_{1}\rangle $ fixes a one-dimensional subspace of $V_{4}$ and thus $ E\sim B_{4}$. Note that $E$ moves. \\
If we consider the subgroup $H:=\langle g_{2},g_{3} \rangle$ and its
associated quotient map $f:C\rightarrow C/H$, all the critical points
in the fibre of $z_{1}$ are critical points of order two for
$f$. Thanks to Riemann-Hurwitz formula we see that the curve $F:=C/H$
is elliptic. Since Fix($H$)=$V_{2}$ we obtain $F\sim B_{2}$. Although
this curve remains constant for the family, we have that
$P(C,E)\sim B_{2}\times B_{4}^{2}$, hence it moves. Thus (37) is
neither contained in a fibre of the Prym map of (6e) nor in a fibre of
the map $\phi$ of (6e).

\qed

\end{document}